\documentclass[10pt]{article}

\input xy

\xyoption{all}

\usepackage{amssymb,amsbsy,amsmath,amsthm,graphicx,epsfig,times}

\newtheorem{thm}{Theorem}[section]
\newtheorem{lem}[thm]{Lemma}
\newtheorem{prop}[thm]{Proposition}
\newtheorem{cor}[thm]{Corollary}

\newtheorem{dfn}[thm]{Definition}

\theoremstyle{remark}

\newtheorem*{rmk}{Remark}
\newtheorem*{ex}{Example}
\newtheorem*{exs}{Examples}

\newcommand{\bs}[1]{\boldsymbol{#1}}

\renewcommand{\rm}[1]{\mathrm{#1}}

\newcommand{\bbN}{\mathbb{N}}

\newcommand{\bbR}{\mathbb{R}}

\newcommand{\sfE}{\mathsf{E}}
\newcommand{\sfP}{\mathsf{P}}

\renewcommand{\d}{\mathrm{d}}

\renewcommand{\P}{\mathcal{P}}

\newcommand{\frH}{\mathfrak{H}}

\renewcommand{\S}{\Sigma}

\renewcommand{\a}{\alpha}
\renewcommand{\b}{\beta}
\newcommand{\eps}{\varepsilon}
\newcommand{\g}{\gamma}

\renewcommand{\l}{\lambda}

\newcommand{\s}{\sigma}

\renewcommand{\Pr}{\mathrm{Pr}}

\renewcommand{\hat}[1]{\widehat{#1}}
\newcommand{\ol}[1]{\overline{#1}}

\newcommand{\fin}{\nolinebreak\hspace{\stretch{1}}$\lhd$}

\newcommand{\actson}{\curvearrowright}
\renewcommand{\to}{\longrightarrow}
\newcommand{\Samp}{\rm{Samp}}

\begin{document}

\title{Exchangeable random measures}

\author{Tim Austin\footnote{Research partially supported by a research fellowship from the Clay Mathematics Institute}\\ \\ \small{Courant Institute, New York University}\\ \small{New York, NY 10012, USA}}

\date{}

\maketitle

\begin{abstract}
Let $A$ be a standard Borel space, and consider the space $A^{\bbN^{(k)}}$ of $A$-valued arrays indexed by all size-$k$ subsets of $\bbN$.  This paper concerns random measures on such a space whose laws are invariant under the natural action of permutations of $\bbN$.  The main result is a representation theorem for such `exchangeable' random measures, obtained using the classical representation theorems for exchangeable arrays due to de Finetti, Hoover, Aldous and Kallenberg.

After proving this representation, two applications of exchangeable random measures are given.  The first is a short new proof of the Dovbysh-Sudakov Representation Theorem for exchangeable positive semi-definite matrices.  The second is in the formulation of a natural class of limit objects for dilute mean-field spin glass models, retaining more information than just the limiting Gram-de Finetti matrix used in the study of the Sherrington-Kirkpatrick model.

\vspace{7pt}

\begin{center}
\textbf{R\'{e}sum\'{e}}
\end{center}

\vspace{7pt}

Soit $A$ un espace de Borel standard, et soit
$A^{\mathbb{N}^{(k)}}$ l'ensemble des tableaux \`a valeur dans $A$ index\'es par
les sous-ensembles de $\mathbb{N}$ de taille $k$. On s'int\'eresse aux mesures
al\'eatoires sur un tel espace dont la loi est invariante par l'action naturelle
des permutations de $\mathbb{N}$. Le r\'esultat principal est une
repr\'esentation de ces mesures al\'eatoires ``\'echangeables'', obtenue \`a
partir des th\'eor\`emes de repr\'esentations classiques de de Finetti, Hoover, Aldous et
Kallenberg pour des tableaux \'echangeables.

Apr\`es avoir prouv\'e cette repr\'esentation, on en donne deux applications. La
premi\`ere est une nouvelle courte preuve du th\'eor\`eme de repr\'esentation de
Dovbysh-Sudakov pour des matrices d\'efinie semi-positive \'echangeables. La seconde concerne la
formulation d'une classe naturelle d'objets limites pour des mod\`eles de champs
moyens dilu\'es pour des verres de spins qui capture plus d'information que la
seule matrice limite de Gram-de Finetti qui est notamment utilis\'ee dans
l'\'etude du mod\`ele de Sherrington-Kirkpatrick.
\end{abstract}

\section{Introduction}

The theory of exchangeable arrays of random variables emerged in work of Hoover~\cite{Hoo79,Hoo82}, Aldous~\cite{Ald81,Ald82,Ald85} and Kallenberg~\cite{Kal89,Kal92}, and amounts to a significant generalization of the classical de Finetti-Hewitt-Savage Theorem on exchangeable sequences.  The heart of the theory is a collection of representation theorems for general such arrays, which then beget more specialized representation results such as the Dovbysh-Sudakov Theorem for exchangeable PSD matrices.

This note will consider the related setting of random measures on spaces of arrays, where now the laws of those random measures are assumed invariant under the relevant group action.  Intuitively, this introduces an `extra layer of randomness'.  In order to introduce these formally, let $[n] := \{1,2,\ldots,n\}$ for $n \in \bbN$, let $S_\bbN = \bigcup_{n\geq 1}S_{[n]}$ be the group of all permutations of $\bbN$ which fix all but finitely many elements, and consider a measurable action $T:S_\bbN \actson E$ on a standard Borel space $E$.  In full, this is a measurable function
\[T:S_\bbN\times E\to E:(\s,x)\mapsto T^\s x\]
such that
\[T^{\rm{id}_\bbN} = \rm{id}_E \quad \hbox{and} \quad T^{\s_1}T^{\s_2}x = T^{\s_2\s_1}x \quad \forall \s_1,\s_2,x.\]
As is standard, if $\mu \in \Pr\,E$ and $\s \in S_\bbN$ then $T^\s_\ast\mu$ denotes the image measure of $\mu$ under $T^\s$.

\begin{dfn}
If $E$ is a standard Borel space and $T:S_\bbN\actson E$ is a measurable action, then an \textbf{exchangeable random measure} (`\textbf{ERM}') on $(E,T)$ is a random variable $\bs{\mu}$ taking values in $\Pr\,E$ such that
\[\bs{\mu} \stackrel{\rm{law}}{=} T^\s_\ast\bs{\mu} \quad \forall \s \in S_\bbN;\]
that is,
\[\bs{\mu}(A) \stackrel{\rm{law}}{=} \bs{\mu}\{x\,:\,T^\s x \in A\} \quad \forall \s \in S_\bbN,\ A \subseteq_\rm{Borel} E.\]
\end{dfn}

These are essentially what ergodic theorists call `quasi-factors'~\cite[Chapter 8]{Gla03}. We will study these for the group actions that underly the theory of exchangeable arrays.  Given a standard Borel space $A$ and $k \in \bbN$, the space of \textbf{$k$-dimensional arrays valued in $A$} is $A^{\bbN^{(k)}}$, where $\bbN^{(k)}$ denotes the set of size-$k$ subsets of $\bbN$.  An element of such a space of arrays will often be denoted by $(x_e)_{|e|=k}$ or similarly.  (In the following, one could focus instead on arrays indexed by ordered $k$-tuples, but we have chosen the symmetric case as it is a little simpler and arises more often in applications.)  The group $S_\bbN$ acts on $A^{\bbN^{(k)}}$ by permuting coordinates in the obvious manner:
\[T^\s((x_e)_{|e| = k}) = (x_{\s(e)})_{|e| = k},\]
where $\s(e) = \{\s(i)\,:\,i\in e\}$.  Slightly more generally, our main results will also allow Cartesian products of such actions over finitely many different $k$.  Thus, our arrays will usually be indexed by the family $\bbN^{(\leq k)}$ of subsets of $\bbN$ of size at most $k$ for some fixed $k$.

\begin{exs}
\emph{(1)}\quad If an exchangeable random measure $\bs{\mu}$ on $(E,T)$ is deterministic, then its constant value must itself be invariant under the action $T$.  In case $E = A^{\bbN^{(k)}}$ with the action above, this means $\bs{\mu}$ is almost surely the law of an exchangeable $A$-valued, $\bbN^{(k)}$-indexed array.

\emph{(2)}\quad On the other hand, if $\mu$ is a $T$-invariant measure for any action $(E,T)$, then another way to obtain an exchangeable random measure from it is to let
\[\bs{\mu} := \delta_X\]
where $X$ is a random element of $E$ with law $\mu$, and $\delta_X$ is the Dirac mass at $X$.

\emph{(3)}\quad In case $E = A^{\bbN^{(k)}}$ with the action above, example (2) fits into a more general family as follows.  The space of probability measures $\Pr\,A$ is also standard Borel with the Borel structure generated by evaluation of measures on Borel sets.  Suppose $(\bs{\l}_e)_{|e| = k}$ is an exchangeable array of $(\Pr\,A)$-valued random variables, and now let
\[\bs{\mu} = \bigotimes_{|e| = k}\bs{\l}_e.\]
This class of examples will feature again later.  Such an example is called an \textbf{exchangeable random product measure} (`\textbf{ERPM}').

\emph{(4)}\quad It is also easy to exhibit an ERM which is not ERPM.  For example, let $\bs{\Pi} = (A,B)$ be a uniform random bipartition of $\bbN$ (this is obviously exchangeable), and having chosen $\bs{\Pi}$ let $\bs{\mu} \in \Pr\{0,1\}^{\bbN^{(2)}}$ be the probability which has two atoms of mass $\frac{1}{2}$ on the points
\[1_{A^{\bbN^{(2)}}} \quad \hbox{and} \quad 1_{B^{\bbN^{(2)}}}.\]

\emph{(5)}\quad Lastly, given a measurable family of exchangeable random measures $\bs{\mu}_t$ indexed by a parameter $t \in [0,1)$, we may average over this parameter to obtain a mixture of these exchangeable random measures:
\[\bs{\mu} = \int_0^1 \bs{\mu}_t\,\d t.\]
This is clearly still exchangeable. \fin
\end{exs}

The main result of this paper characterizes all ERMs on spaces of arrays. To motivate it, we next recall the Representation Theorem for exchangeable arrays themselves.  This requires some more notation.

First, for any set $S$ we let $\P S$ denote the power set of $S$.

Next, suppose that $B_0$, $B_1$, \ldots, $B_k$ and $A$ are standard Borel spaces.  A Borel function
\[f:B_0\times B_1^k\times B_2^{[k]^{(2)}}\times \cdots\times B_k = \prod_{i\leq k}B_i^{[k]^{(i)}}\to A\]
is \textbf{middle-symmetric} if
\[f\big(x,(x_i)_{i \in [k]},(x_a)_{a \in [k]^{(2)}},\ldots,x_{[k]}\big) =
f\big(x,(x_{\s(i)})_{i\in [k]},(x_{\s(a)})_{a \in [k]^{(2)}},\ldots,x_{[k]}\big)\]
for all $\s\in S_{[k]}$.

Given standard Borel spaces $B_0$, $B_1$, \ldots, $B_k$ and $A_0$, $A_1$, \ldots, $A_k$, and middle-symmetric Borel functions
\[f_i:\prod_{j\leq i}B_j^{[i]^{(j)}}\to A_i, \quad i=0,1,\ldots,k,\]
we will write $\hat{f}$ for the function
\[\prod_{i\leq k}B_i^{[k]^{(i)}}\to \prod_{i\leq k} A_i^{[k]^{(i)}}:(x_e)_{e\subseteq [k]}\mapsto \big(f_{|e|}((x_a)_{a\subseteq e})\big)_{e \subseteq [k]},\]
which combines all of the $f_i$.

The tuple $(f_0,\ldots,f_k)$ is referred to as a \textbf{skew-product tuple}, and the associated function $\hat{f}$ as a function of \textbf{skew-product type}; clearly the latter determines the former uniquely.

\begin{ex}
If $k=2$, then a function of skew-product type $[0,1)^{\P[2]}\to [0,1)$ takes the form
\[\hat{f}(x,x_1,x_2,x_{12}) = (f_0(x),f_1(x,x_1),f_1(x,x_2),f_2(x,x_1,x_2,x_{12})).\]
\fin
\end{ex}

It is easily checked that if $\hat{f}$ and $\hat{g}$ are functions of skew-product type for the same $k$, then so is $\hat{g}\circ \hat{f}$.  In terms of $(f_0,\ldots,f_k)$ and $(g_0,\ldots,g_k)$ this composition corresponds to the skew-product tuple
\[h_i((x_a)_{a \subseteq [i]}) := g_i\big(\big(f_{|a|}((x_b)_{b\subseteq a})\big)_{a \subseteq [i]}\big), \quad i=0,1,\ldots,k.\]

Slightly abusively, we will also write $\hat{f}$ for the related function
\[\prod_{i\leq k}B_i^{\bbN^{(i)}}\to \prod_{i\leq k}A_i^{\bbN^{(i)}}:(x_e)_{|e| \leq k}\mapsto \big(f_{|e|}((x_a)_{a\subseteq e})\big)_{|e| \leq k},\]
which also determines $(f_0,\ldots,f_k)$ uniquely.

\begin{thm}[Representation Theorem for Exchangeable Arrays; Theorem 7.22 in~\cite{Kal05}]\label{thm:Kal1}
Suppose that $A_0$, $A_1$, \ldots, $A_k$ are standard Borel spaces and that $(X_e)_{|e| \leq k}$ is an exchangeable random array of r.v.s with each $X_e$ valued in $A_{|e|}$.  Then there are middle-symmetric Borel functions
\[f_i:[0,1)^{\P[i]}\to A_i, \quad i=0,1,\ldots,k,\]
such that
\[(X_e)_{|e|\leq k} \stackrel{\rm{law}}{=} \big(f_{|e|}((U_a)_{a \subseteq e})\big)_{|e|\leq k} \stackrel{\rm{dfn}}{=} \hat{f}((U_e)_{|e|\leq k}),\]
where $(U_e)_{|e| \leq k}$ is an i.i.d. family of $\rm{U}[0,1)$-r.v.s. \qed
\end{thm}

The companion Equivalence Theorem, which addresses the non-uniqueness of the representing function $\hat{f}$, will be recalled later.

To produce a random measure, the idea will simply be to use directing functions $f_i$ that depend on two sources of randomness, and then condition on one of them.

\vspace{7pt}

\noindent\textbf{Theorem A}\quad \emph{Suppose that $\bs{\mu}$ is an ERM on $A_0\times \cdots \times A_k^{\bbN^{(k)}}$.  Then there are middle-symmetric Borel functions
\[f_i:([0,1)\times [0,1))^{\P[i]}\to A_i\]
such that
\[\bs{\mu}(\cdot) \stackrel{\rm{law}}{=} \sfP\big(\hat{f}((U_e,V_e)_{|e|\leq k}) \in \cdot\,\big|\,(U_e)_{|e|\leq k}\big),\]
where $U_e$ and $V_e$ for $e \subseteq \bbN$, $|e| \leq k$ are all i.i.d. $\sim \rm{U}[0,1)$. On the right-hand side, this is a measure-valued random variable as a function of the r.v.s $(U_e)_{|e|\leq k}$. }

\vspace{7pt}

We will find that after some manipulation of the problem, Theorem A can be deduced from the Representation Theorem and Equivalence Theorems for exchangeable random arrays themselves.

The proof of Theorem A can be considerably simplified when $k=1$, so we will first prove that case separately.  In that case, the structure given by Theorem A is essentially a combination of examples (3) and (5) above.  To see this, we reformulate the result as follows.

Given a standard Borel space $A$, let $B([0,1),\Pr\,A)$ denote the space of Lebesgue-a.e. equivalence classes of measurable functions $[0,1)\to \Pr\,A$.  Then $B([0,1),\Pr\,A)$ has a natural measurable structure generated by the functionals
\[f \mapsto \int_0^1\phi(t)f(t,B)\,\d t\]
corresponding to all $\phi \in L^\infty[0,1)$ and Borel subsets $B \subseteq A$.  This measurable structure is also standard Borel: for instance, if one realizes $A$ as a Borel subset of a compact metric space, then the above becomes the Borel structure of the topology of convergence in probability on $B([0,1),\Pr\,A)$, which is Polish.

\vspace{7pt}

\noindent\textbf{Theorem B}\quad \emph{If $\bs{\mu}$ is an ERM on $A^\bbN$, then there is an exchangeable sequence of r.v.s $(\bs{\l}_i)_{i \in \bbN}$ taking values in $B([0,1),\Pr\,A)$ such that
\[\bs{\mu}(\cdot) \stackrel{\rm{law}}{=} \int_0^1 \big(\bigotimes_{i \in \bbN} \bs{\l}_i(t,\cdot)\big)\ \d t.\] }

\vspace{7pt}

So when $k=1$, every ERM is a mixture of ERPMs.

With the structure given by Theorem B, one may next apply the de Finetti-Hewitt-Savage Theorem to the sequence $\bs{\l}_i$ to obtain a random measure $\bs{\g}$ on $B([0,1),\Pr\,A)$ such that $\bs{\l}_i$ is obtained by first choosing $\bs{\g}$ and then choosing $\bs{\l}_i$ i.i.d. with law $\bs{\g}$.  We write $\Samp(\bs{\g})$ for the ERM obtained by this procedure, and refer to $\bs{\g}$ as a \textbf{directing random measure} for $\bs{\mu}$.

After proving Theorems A and B, we offer a couple of applications of the case $k=1$.  These applications can also be given higher-dimensional extensions using the cases $k\geq 2$, but those extensions seem less natural.  The reader interested only in the applications need not read the proof of the general case of Theorem A.

The first application is a new proof of the classical Dovbysh-Sudakov Theorem:

\vspace{7pt}

\noindent\textbf{Dovbysh-Sudakov Theorem}\quad \emph{Suppose $(R_{ij})_{i,j \in \bbN}$ is a random matrix which is a.s. positive semi-definite, and is exchangeable in the sense that
\[(R_{\s(i)\s(j)})_{i,j} \stackrel{\rm{law}}{=} (R_{ij})_{i,j} \quad \forall \s \in S_\bbN.\]
Then there are a separable real Hilbert space $\frH$ and an exchangeable sequence $(\xi_i,a_i)_{i\in \bbN}$ of random variables valued in $\frH\times [0,\infty)$ such that
\[(R_{ij})_{i,j} \stackrel{\rm{law}}{=} (\langle \xi_i,\xi_j\rangle + \delta_{ij}a_i)_{i,j},\]
where $\delta_{ij}$ is the Kronecker delta.}

\vspace{7pt}

This first appeared in~\cite{DovSud82}, and more complete accounts were given in~\cite{Hes86} and~\cite{Pan09}.  The proofs of Hestir and Panchenko start with the Aldous-Hoover Representation Theorem, which treats $(R_{ij})_{i,j}$ as a general two-dimensional exchangeable array.  They then require several further steps to show that the PSD assumption implies a simplification of that general Aldous-Hoover representation into the form promised above.  On the other hand, we will find that if one simply interprets $(R_{ij})_{i,j}$ as the covariance matrix of an exchangeable random measure, then one can read off the Dovbysh-Sudakov Theorem from Theorem B, which in turn does not require the Aldous-Hoover Theorem.

Our second application is to the study of certain mean-field spin glass models, and particularly Viana and Bray's dilute version of the Sherrington-Kirkpatrick model~\cite{ViaBra85}.  In the case of the original Sherrington-Kirkpatrick model a great deal has now been proven, much of it relying on the notions of `random overlap structures' and their directing random Hilbert space measures: see, for instance, Panchenko's monograph~\cite{Pan--book}.  The analogous theory for dilute models is less advanced.  In this note we will simply sketch how the main conjecture of Replica Symmetry Breaking can be formulated quite neatly in terms of limits of exchangeable random measures, translating from the earlier works~\cite{PanTal04,Pan--dil}.  We will not recall most of the spin glass theory behind this conjecture, but will refer the reader to those references for more background.

\subsubsection*{Acknowledgements}

I am grateful to Kavita Ramanan, Dmitry Panchenko and the anonymous referees for several helpful suggestions and references.

\section{The replica trick}

The key to Theorem A is the simple observation that the law of a random measure on some space $E$ can be equivalently described by the law of a random sequence in $E$, obtained by first sampling and quenching that random measure, and then sampling i.i.d. from it.  This idea is standard in the more general setting of representing quasi-factors in ergodic theory~(\cite[Chapter 8]{Gla03}).  In a sense, it is an abstract version of the `replica trick' from the statistical physics of spin glasses~(\cite{MezParVir--book}).  In physics, the phrase `replica trick' usually refers to the calculation of the sequence of moments of the (random) partition function of a random Gibbs measure, which is then fed into an ansatz for guessing more about the law of the partition function, such as the expected free energy.  This resembles our `replica trick' insofar as computing a moment of the random partition function amounts to computing the partition function for the law of several i.i.d. samples from the random Gibbs measure.

Before we proceed, first observe that, since any standard Borel space is isomorphic to a Borel subset of a compact metric space, we may replace the spaces $A_0$, \ldots, $A_k$ with such enveloping compact spaces in Theorems A and B, and so assume these spaces are themselves compact.  We will make this assumption throughout the proofs of those theorems, although some non-compact examples will re-appear later in the applications.

\begin{prop}[Replica trick]\label{prop:replicas}
If $\bs{\mu}$ is an ERM on $\prod_{i\leq k}A_i^{\bbN^{(i)}}$, then there are auxiliary standard Borel spaces $\ol{A}_0$, $\ol{A}_1$, \ldots, $\ol{A}_k$ and an exchangeable array $(Y_e,X_e)_{|e| \leq k}$ of random variables such that
\begin{itemize}
\item each $(Y_e,X_e)$ takes values in $\ol{A}_{|e|}\times A_{|e|}$, and
\item one has
\[\bs{\mu}(\cdot) \stackrel{\rm{law}}{=} \sfP((X_e)_{|e| \leq k} \in \cdot\,|\,(Y_e)_{|e| \leq k}).\]
\end{itemize}
\end{prop}

\begin{proof}
After enlarging the background probability space if necessary, we may couple the random variable $\bs{\mu}$ with a doubly-indexed family of random variables
\begin{eqnarray}\label{eq:big-family}
\big((X_{i,e})_{i\in\bbN,e \in \bbN^{(\leq k)}},(X_e)_{e \in \bbN^{(\leq k)}}\big),
\end{eqnarray}
all taking values in one of the $A_i$s, as follows:
\begin{itemize}
\item first, sample the random measure $\bs{\mu}$ itself;
\item then, choose the sub-families $(X_e)_{|e|\leq k}$, $(X_{1,e})_{|e| \leq k}$, $(X_{2,e})_{|e| \leq k}$, \ldots independently with law $\bs{\mu}$.
\end{itemize}

In notation, this coupling is defined by
\begin{multline*}
\sfP\big((X_e)_{|e| \leq k} \in \d\bs{a},\,(X_{1,e})_{|e| \leq k} \in \d \bs{a}_1,\,(X_{2,e})_{|e| \leq k} \in \d \bs{a}_2,\,\ldots\,\big|\,\bs{\mu}\big)\\
= \bs{\mu}(\d \bs{a})\cdot \bs{\mu}(\d \bs{a}_1)\cdot \bs{\mu}(\d \bs{a}_2)\cdot \cdots.
\end{multline*}

Having done this, let $\ol{A}_i := A_i^\bbN$ and let $Y_e := (X_{j,e})_{j\in \bbN} \in \ol{A}_{|e|}$ for each $e \in \bbN^{(\leq k)}$.  The exchangeability of $\bs{\mu}$ implies that the joint distribution of the family~(\ref{eq:big-family}) is invariant under applying elements of $S_\bbN$ to the indexing sets $e$, and hence that the process $(Y_e,X_e)_{|e|\leq k}$ is exchangeable.

On the other hand, since we assume each $A_i$ is compact, so is $\prod_{i \leq k}A_i^{\bbN^{(i)}}$, and now the Law or Large Numbers shows that in the above process one has the a.s. convergence of empirical measures
\[\frac{1}{N}\sum_{n=1}^N \delta_{(X_{n,e})_{|e|\leq k}} \to \bs{\mu}\]
in the vague topology on $\Pr\prod_{i\leq k}A_i^{\bbN^{(i)}}$.

Therefore in the process
\[\big((Y_e)_{|e|\leq k},(X_e)_{|e|\leq k},\bs{\mu}\big)\]
the family of r.v.s $(Y_e)_{|e|\leq k}$ determine $\bs{\mu}$ a.s., whereas conditionally on $\bs{\mu}$ the family $(Y_e)_{|e|\leq k}$ becomes independent from $(X_e)_{|e|\leq k}$.  This implies that
\[\sfP\big((X_e)_{|e| \leq k}\in \cdot\,\big|\,(Y_e)_{|e|\leq k}\big) = \sfP\big((X_e)_{|e| \leq k}\in \cdot\,\big|\,\bs{\mu}\big) = \bs{\mu}(\cdot)\quad \hbox{a.s.},\]
as required.
\end{proof}

\section{Proofs in one dimension}

\subsection{Some preliminaries}

We will repeatedly need the following standard tool from measure-theoretic probability.  See, for instance, the slightly-stronger Theorem 6.10 in~\cite{Kal02}.

\begin{lem}[Noise-Outsourcing Lemma]\label{lem:NOL}
Suppose that $A$ and $B$ are standard Borel spaces and that $(X,Y)$ is an $(A\times B)$-valued r.v.  Then, possibly after enlarging the background probability space, there are a r.v. $U\sim \rm{U}[0,1)$ coupled with $X$ and $Y$ and a Borel function $f:A\times [0,1)\to Y$ such that $U$ is independent from $X$ and 
\[(X,Y) = (X,f(X,U)) \quad \hbox{a.s.}.\] \qed
\end{lem}

Of course, the function $f$ in this lemma is highly non-unique.  The degenerate case in which $X$ is deterministic is still important: it reduces to the assertion that for any standard Borel probability space $(B,\nu)$ there is a Borel function $f:[0,1)\to B$ such that $f(U) \sim \nu$ when $U \sim \rm{U}[0,1)$.

Finally, let us recall the full de Finetti-Hewitt-Savage Theorem for the case $k=1$, which is rather stronger than just the case $k=1$ of Theorem~\ref{thm:Kal1}.  The following is the combination of Proposition 1.4 and Corollaries 1.5 and 1.6 in~\cite{Kal05}.

\begin{thm}\label{thm:deFHS}
Suppose $A$ is a compact metric space and $(X_n)_n$ is an exchangeable sequence of $A$-valued r.v.s.  Then the sequence of empirical distributions
\[W_N := \frac{1}{N}\sum_{n=1}^N \delta_{X_n} \in \Pr\,A\]
converges a.s. to a $(\Pr\,A)$-valued r.v. $W$ which has the following properties:
\begin{itemize}
\item[(i)] $W$ is a.s. a function of $(X_n)_n$,
\item[(ii)] $W$ generates the tail $\s$-algebra of $(X_n)_n$ up to $\mu$-negligible sets;
\item[(iii)] the r.v.s $X_n$ are conditionally i.i.d. given $W$;
\item[(iv)] if $Z$ is any other r.v. on the same probability space such that
\[(Z,X_1,X_2,\ldots) \stackrel{\rm{law}}{=} (Z,X_{\s(1)},X_{\s(2)},\ldots) \quad \forall \s \in S_\bbN,\]
then $Z$ is conditionally independent from $(X_n)_n$ over $W$. \qed
\end{itemize}
\end{thm}

\subsection{Proofs in one dimension}

\begin{proof}[Proof of Theorem A in one dimension]
Suppose $\bs{\mu}$ is an ERM on $A^\bbN$ and let $(Y_n,X_n)_n$ be a process as given by Proposition~\ref{prop:replicas}.

We next apply Theorem~\ref{thm:deFHS} twice: first to the sequence $(Y_n)_n$, to obtain a r.v. $W$ taking values in $E := \Pr\,\ol{A}$; and secondly to $(Y_n,X_n)_n$, to obtain a r.v. $Z$ taking values in $F := \Pr\,(\ol{A}\times A)$.  From their definitions as limits of empirical distributions, $W$ is almost surely a function of $Z$.  On the other hand, property (iv) of Theorem~\ref{thm:deFHS} gives that $Z$ is conditionally independent from $(Y_n)_n$ over $W$.

Now pick an $n \in \bbN$. By Lemma~\ref{lem:NOL}, there is a Borel function $f_1:E\times F\times \ol{A}\times [0,1)\to A$ such that
\[(W,Z,Y_n,X_n) \stackrel{\rm{law}}{=} (W,Z,Y_n,f_1(W,Z,Y_n,V_n)),\]
where $(V_n)_n$ are i.i.d. $\sim \rm{U}[0,1)$ and are independent from $(W,Z,Y_n)$.  Moreover, this same $f_1$ works for every $n$, by exchangeability.  It follows that in fact
\begin{eqnarray}\label{eq:first-law-eq}
\big(W,Z,(Y_n,X_n)_{n\in\bbN}\big) \stackrel{\rm{law}}{=} \big(W,Z,(Y_n,f_1(W,Z,Y_n,V_n))_{n\in\bbN}\big),
\end{eqnarray}
because both sides have the same marginals for individual $n$, and both sides are conditionally i.i.d. over $(W,Z)$, so all finite-dimensional marginals agree.

Next, another appeal to Lemma~\ref{lem:NOL} gives a Borel function $g:E\times [0,1)\to F$ such that
\[(W,Z) \stackrel{\rm{law}}{=} (W,g(W,V))\]
with a new independent $V \sim \rm{U}[0,1)$.  This implies that
\[(W,Z,(Y_n)_{n\in \bbN}) \stackrel{\rm{law}}{=} (W,g(W,V),(Y_n)_{n\in\bbN}),\]
because $Z$ is conditionally independent from $(Y_n)_n$ over $W$, so again all finite-dimensional marginals agree.
Combining this with~(\ref{eq:first-law-eq}) gives
\[(W,Z,(Y_n,X_n)_{n\in\bbN}) \stackrel{\rm{law}}{=} (W,g(W,V),(Y_n,f_2(W,Y_n,V,V_n))_{n\in\bbN}),\]
where
\[f_2(w,y,v,v') := f_1(w,g(w,v),y,v').\]

It follows that
\[\sfP((X_n)_n \in \cdot\,|\,(Y_n)_n) = \sfP\big((f_2(W,Y_n,V,V_n))_n \in \cdot\,\big|\,(Y_n)_n\big).\]
Finally, we may apply de Finetti's Theorem again, this time to $(Y_n)_n$, to obtain Borel functions $h_1:[0,1)\to E$ and $h_2:[0,1)^2\to \ol{A}$ such that
\[(W,(Y_n)_{n\in\bbN}) \stackrel{\rm{law}}{=} \big(h_1(U),(h_2(U,U_n))_{n\in\bbN}\big),\]
where $U$ and $(U_n)_n$ are i.i.d. $\sim \rm{U}[0,1)$ r.v.s, independent of everything else.  Letting
\[f(u,u',v,v') := f_2(h_1(u),h_2(u,u'),v,v'),\]
substituting for $(W,(Y_n)_n)$ in the above gives
\[\sfP((X_n)_n \in \cdot\,|\,(Y_n)_n) \stackrel{\rm{law}}{=} \sfP\big((f(U,U_n,V,V_n))_n \in \cdot\,\big|\,U,(U_n)_n\big),\]
as required.
\end{proof}

\begin{proof}[Proof of Theorem B]
By the Law of Iterated Conditional Expectation, the representation obtained above may be re-written as
\begin{eqnarray*}
\bs{\mu}(\cdot) &\stackrel{\rm{law}}{=}& \sfE\Big(\sfP\big((f(U,U_n,V,V_n))_n \in \cdot\,\big|\,U,(U_n)_n,V\big)\,\Big|\,U,(U_n)_n\Big)\\
&=& \sfE\Big(\bigotimes_{n\in\bbN}\sfP\big(f(U,U_n,V,V_n) \in \cdot\,\big|\,U,U_n,V\big)\,\Big|\,U,(U_n)_n\Big)\\
&=& \int_0^1 \Big( \bigotimes_{n \in \bbN}\bs{\l}_n(t,\,\cdot\,)\Big)\,\d t,
\end{eqnarray*}
where
\[\bs{\l}_n(t,\,\cdot\,) := \sfP(f(U,U_n,t,V_n) \in \cdot\,|\,U,U_n).\]
As functions of $(U,U_n)$ for each $n$, these form an exchangeable sequence of r.v.s taking values in $B([0,1),\Pr\,A)$, so the proof is complete.
\end{proof}

\subsection{Relation to row-column exchangeability}\label{subs:RCE}

A relative of exchangeability for a two-dimensional random array $(X_{i,n})_{(i,n)\in\bbN^2}$ is \textbf{row-column exchangeability}, which asserts that
\[(X_{\s(i),\tau(n)})_{i,n} \stackrel{\rm{law}}{=} (X_{i,n})_{i,n} \quad \forall \s,\tau \in S_\bbN.\]
Since $\s$ and $\tau$ may be chosen separately, this is a rather stronger symmetry than ordinary two-dimensional exchangeability.  Here, too, there is a representation theorem due to Aldous and Hoover, and also a version in arbitrary dimensions due to Kallenberg, who calls such arrays `separately exchangeable'.

\begin{thm}[Corollary 7.23 in~\cite{Kal05}]\label{thm:RCE}
If $(X_{i,n})_{i,n}$ is an $A$-valued row-column exchangeable array then there is a Borel function $[0,1)^4 \to A$ such that
\[(X_{i,n})_{i,n} \stackrel{\rm{law}}{=} (f(Z,U_i,V_n,W_{i,n}))_{i,n},\]
where $Z$, $U_i$ for $i\in\bbN$, $V_n$ for $n \in \bbN$ and $W_{i,n}$ for $(i,n) \in \bbN^2$ are i.i.d. $\sim \rm{U}[0,1)$. \qed
\end{thm}

An alternative proof of Theorem B can be given via Theorem~\ref{thm:RCE}.  One begins with the construction of the two-dimensional random array $(X_{i,n})_{i,n}$ as in the proof of Proposition~\ref{prop:replicas} (where the sets $e$ have become singletons $n$). Since this array is row-column exchangeable, the Representation Theorem gives
\[(X_{i,n})_{i,n} \stackrel{\rm{law}}{=} (f(U,U_i,V_n,W_{i,n}))_{i,n}\]
for some Borel directing function $f:[0,1)^4\to A$, where $U$, $U_i$ for $i\in\bbN$, $V_n$ for $n \in \bbN$ and $W_{i,n}$ for $i,n \in \bbN$ are i.i.d. $\sim\rm{U}[0,1)$.
One can now read off a directing random measure $\g(U)$ on $B([0,1),\Pr\,A)$, a function of $U \sim \rm{U}[0,1)$, in the following two steps:
\begin{itemize}
\item first, for each fixed $U$ and $U'$ one obtains an element $\l(U,U') \in B([0,1),\Pr\,A)$ according to
\[\l(U,U')(t,\,\cdot\,) = \sfP_W(f(U,U',t,W) \in \cdot\,),\quad W \sim \rm{U}[0,1);\]
\item second, $\g(U)$ is the distribution of $\l(U,U')$ where $U' \sim \rm{U}[0,1)$.
\end{itemize}
On the other hand, a couple of simple applications of the Noise-Outsourcing Lemma show that any directing random measure $\g$ on $B([0,1),\Pr\,A)$ can be represented this way, so this gives a bijective correspondence
\begin{multline*}
\{\hbox{directing random measures on}\ B([0,1),\Pr\,A)\ \hbox{up to equivalence}\} \\ \leftrightarrow \{\hbox{directing functions}\ [0,1)^4\to A\ \hbox{up to equivalence}\},
\end{multline*}
where `up to equivalence' refers to the possibility that different directing random measures or directing functions may give rise to the same row-column exchangeable array.

This approach is the basis of the paper~\cite{Pan--dil}, to be discussed later.  It is quick, but at the expense of assuming Theorem~\ref{thm:RCE}.

On the other hand, our approach to Theorem B does not use any exchangeability theory in dimensions greater than one. Moreover, one can reverse the idea above to give a fairly quick proof of Theorem~\ref{thm:RCE} using Theorem B.

\begin{proof}[Proof of Theorem~\ref{thm:RCE} from Theorem B]
First let $(X_{i,n})_{i,n} = ((X_{i,n})_i)_n$, thought of as an exchangeable sequence of $A^\bbN$-valued r.v.s.  By the de Finetti-Hewitt-Savage Theorem applied to the exchangeability in $n$, its law is a mixture of product measures; equivalently, there is a $(\Pr\,A^\bbN)$-valued r.v. $\bs{\mu}$ such that
\[\rm{law}((X_{i,n})_i)_n) = \sfE(\bs{\mu}^{\otimes \bbN}).\]
On the other hand, for any $\s \in S_\bbN$ the exchangeability in $i$ gives
\[\sfE(\bs{\mu}^{\otimes \bbN}) = \rm{law}((X_{i,n})_i)_n)  = \rm{law}((X_{\s(i),n})_i)_n) = \sfE((T^\s_\ast\bs{\mu})^{\otimes \bbN}),\]
where $T^\s:A^\bbN\to A^\bbN$ is the corresponding coordinate-permuting transformation.  By the uniqueness of the de Finetti-Hewitt-Savage decomposition, this implies
\[\bs{\mu} \stackrel{\rm{law}}{=} T^\s_\ast\bs{\mu} \quad \forall \s \in S_\bbN,\]
so $\bs{\mu}$ is an ERM.  Therefore Theorem B gives
\[\rm{law}((X_{i,n})_i)_n) = \sfE\Big(\int_0^1\bigotimes_{i \in \bbN}\bs{\l}_i(t,\cdot)\,\d t\Big)^{\otimes \bbN}\]
for some exchangeable random sequence $(\bs{\l}_i)_i$ taking values in $B([0,1),\Pr\,A)$.

Next, applying the de Finetti-Hewitt-Savage Theorem to the sequence $(\bs{\l}_i)_i$ itself gives a Borel function $F:[0,1)^2\to B([0,1),\Pr\,A)$ such that the above becomes
\begin{multline}\label{eq:law-of-RCE}
\rm{law}((X_{i,n})_i)_n) = \sfE_{Z,(U_i)_i}\Big(\int_0^1\bigotimes_{i \in \bbN}F(Z,U_i,t,\cdot)\,\d t\Big)^{\otimes \bbN}\\
= \sfE_{Z,(U_i)_i,(V_n)_n}\bigotimes_{(i,n)\in\bbN^2}F(Z,U_i,V_n,\cdot),
\end{multline}
where $Z$, $U_i$ for $i\in\bbN$ and $V_n$ for $n \in \bbN$ are i.i.d. $\sim \rm{U}[0,1)$, and in the second equality we have simply changed notation from `$\int_0^1 \cdot\, \d t$' to `$\sfE_{V_n}$'.

Finally, by Lemma~\ref{lem:NOL} there is a Borel function $f:[0,1)^4\to A$ such that
\[\sfP(f(Z,U,V,W)\in\, \cdot\,|\,Z,U,V) = F(Z,U,V,\cdot) \quad \hbox{a.s.}\]
when $Z$, $U$, $V$, $W$ are i.i.d. $\sim \rm{U}[0,1)$, and now the right-hand side of~(\ref{eq:law-of-RCE}) becomes $\rm{law}((f(U,U_i,V_n,W_{i,n})_i)_n)$, as required.
\end{proof}

\section{Proof in higher dimensions}

Not all parts of Theorem~\ref{thm:deFHS} generalize to higher-dimensional arrays, and instead we must make a more careful argument using the Equivalence Theorem~\ref{thm:Kal2} below.

\subsection{Some more preliminaries}

The Equivalence Theorem characterizes when two functions direct the same process in the setting of Theorem~\ref{thm:Kal1}.  Its formulation needs the following notion.  Let $C = [0,1)^r$ and $D = [0,1)^s$ for some $r,s \in \bbN$.  Then a skew-product tuple $(f_0,\ldots,f_k)$ in which each $f_i:\prod_{j\leq i}C^{[i]^{(j)}} \to D$ for each $i$ gives rise to a skew-product-type function $\hat{f}:C^{\P[k]}\to D^{\P[k]}$, which is a map between Euclidean cubes.  We will write that the skew-product tuple is \textbf{Lebesgue-measure-preserving} if for all $i=0,1,\ldots,k$ and all $(x_a)_{a\subsetneqq [i]} \in C^{\P[i]\setminus [i]}$, one has
\[U\sim \rm{U}(C) \quad \Longrightarrow \quad f_i\big((x_a)_{a \subsetneqq [i]},U\big) \sim \rm{U}(D).\]
This implies, in particular, that $\hat{f}$ pushes Lebesgue measure on $C^{\P[k]}$ to Lebesgue measure on $D^{\P[k]}$ (although these cubes have different dimensions if $r\neq s$).  However, the assertion that the skew-product tuple is Lebesgue-measure-preserving can be strictly stronger than this, in case the functions $f_i$ are not all injective.

The Equivalence Theorem is as follows.

\begin{thm}[Equivalence Theorem for directing functions; Theorem 7.28 in~\cite{Kal05}]\label{thm:Kal2}
If $\hat{f},\hat{f}':[0,1)^{\P[k]}\to A_0\times \cdots \times A_i^{[k]^{(i)}}\times \cdots \times A_k$ are functions of skew-product type such that
\[\big(f_{|e|}((U_a)_{a \subseteq e})\big)_{|e|\leq k} \stackrel{\rm{law}}{=} \big(f_{|e|}'((U_a)_{a \subseteq e})\big)_{|e| \leq k},\]
then there are functions $\hat{G},\hat{G'}:[0,1)^{\P[k]}\to [0,1)^{\P[k]}$ of skew-product type, whose skew-product tuples are Lebesgue-measure-preserving, and which make the following diagram commute:
\begin{center}
$\phantom{i}$\xymatrix{
& [0,1)^{\P[k]}\ar_{\hat{G}}[dl]\ar^{\hat{G'}}[dr]\\
[0,1)^{\P[k]}\ar_{\hat{f}}[dr] && [0,1)^{\P[k]}\ar^{\hat{f}'}[dl]\\
& A_0\times A_1^k\times \cdots \times A_k.
}
\end{center}
\end{thm}

In connection with this theorem, we will also need the following `factorization' result.

\begin{cor}\label{cor:rep-skew-prod}
Let $U_{\subseteq [k]} = (U_e)_{e\subseteq [k]}$ and $V_{\subseteq [k]}$ be independent uniform r.v.s valued in $[0,1)^{\P[k]}$.  If
\[G:[0,1)^{\P[k]}\to [0,1)^{\P[k]}\]
is a function of skew-product type whose skew-product tuple is Lebesgue-measure-preserving, then there is another function
\[H:([0,1)\times [0,1))^{\P[k]}\to [0,1)^{\P[k]}\]
of skew-product type, whose skew-product tuple is Lebesgue-measure-preserving, and such that
\[U_{\subseteq [k]} = G\big(H(U_{\subseteq [k]},V_{\subseteq [k]})\big) \quad \hbox{a.s.}.\]
\end{cor}

Another way to express this is that the maps in the following diagram come from Lebesgue-measure-preserving skew-product tuples and a.s. commute:
\begin{center}
$\phantom{i}$\xymatrix{
([0,1)\times [0,1))^{\P[k]}\ar_{\Pi}[drr]\ar^-H[rr] && [0,1)^{\P[k]}\ar^G[d]\\
&& [0,1)^{\P[k]},
}
\end{center}
where
\[\Pi\big((x_e,y_e)_{e\subseteq [k]}\big) = (x_e)_{e\subseteq [k]}\]
is the obvious projection.

Geometrically, the intuition here is that $G$ is `almost onto' (since its image measure is Lebesgue), and that as a result one can represent it as the projection map $\Pi$ after using $H$ to `straighten out the fibres'.

\begin{proof}
Let $G$ be defined by the skew-product tuple $(G_0,\ldots,G_k)$. We must construct the skew-product tuple $(H_0,\ldots,H_k)$ that defines $H$.  In terms of these tuples, our requirement is that
\begin{eqnarray}\label{eq:g-h-eq}
G_i\big(\big(H_{|e|}((U_a,V_a)_{a\subseteq e})\big)_{e\subseteq [i]}\big) = U_{[i]} \quad \hbox{a.s.} \quad \forall i = 0,1,\ldots,k.
\end{eqnarray}

When $i=0$ this simplifies to
\[G_0(H_0(U_0,V_0)) = U_0 \quad \hbox{a.s.}.\]
We can obtain such an $H_0$ from the Noise-Outsourcing Lemma~\ref{lem:NOL} as follows.  Let $Z_0$ be a $\rm{U}[0,1)$-r.v. and let $X_0 := G_0(Z_0)$, so this is also $\sim\rm{U}[0,1)$.  Applying Lemma~\ref{lem:NOL} to the pair $(X_0,Z_0)$ gives a Borel function $H_0:[0,1)\times [0,1)\to [0,1)$ such that
\[(X_0,Z_0) = (X_0,H_0(X_0,Y_0)) \quad \hbox{a.s.}\]
for some $Y_0 \sim \rm{U}[0,1)$ independent from $X_0$.  Since $X_0 = G_0(Z_0)$, applying $G_0$ to the second coordinates here gives
\[X_0 = G_0(H_0(X_0,Y_0)) \quad \hbox{a.s.}\]

The general case now follows by induction on $i$. Suppose that $i\geq 1$, let $Y_e$ for $e \subsetneqq [i]$ and $Z_e$ for $e \subseteq [i]$ be i.i.d. $\sim \rm{U}[0,1)$; define $X_e := G_{|e|}((Z_a)_{a\subseteq e})$ for all $e \subseteq [i]$; and assume that $H_0$, \ldots, $H_{i-1}$ have already been constructed such that
$Z_e := H_{|e|}((X_a,Y_a)_{a\subseteq e})$ for each $e \subsetneqq [i]$. Applying Lemma~\ref{lem:NOL} again gives a Borel function $H_i:([0,1)\times [0,1))^{\P[i]}\to [0,1)$ and a r.v. $Y_{[i]}\sim \rm{U}[0,1)$ such that
\[\big((X_e)_{e\subseteq [i]},(Y_e)_{e\subsetneqq [i]},Z_{[i]}\big) = \big((X_e)_{e\subseteq [i]},(Y_e)_{e\subsetneqq [i]},H_i\big((X_e)_{e\subseteq [i]},(Y_e)_{e\subseteq [i]}\big)\big),\]
and as before this is equivalent to the desired equality~(\ref{eq:g-h-eq}).
\end{proof}

\subsection{Completion of the proof}

We need the following enhancement of Proposition~\ref{prop:replicas}.

\begin{lem}\label{lem:couple-to-Us}
If $\bs{\mu}$ is an ERM on $\prod_{i\leq k}A_i^{\bbN^{(i)}}$, then there is an exchangeable array $(U_e,X_e)_{|e| \leq k}$ such that
\begin{itemize}
\item $(U_e)_{|e| \leq k}$ are i.i.d. $\sim\rm{U}[0,1)$,
\item each $X_e$ takes values in $A_{|e|}$,
\item and one has
\[\bs{\mu}(\cdot) \stackrel{\rm{law}}{=} \sfP\big((X_e)_{|e|\leq k}\in \cdot\,\big|\,(U_e)_{|e| \leq k}\big).\]
\end{itemize}
\end{lem}

\begin{proof}
Let $(Y_e,X_e)_{|e|\leq k}$ be the process given by Proposition~\ref{prop:replicas}, with each $(Y_e,X_e)$ taking values in $\ol{A}_{|e|}\times A_{|e|}$.  By the Representation Theorem~\ref{thm:Kal1} applied to $(Y_e)_{|e|\leq k}$, there is a function $\hat{f}:[0,1)^{\P[k]}\to \prod_{i\leq k}\ol{A}_i^{[k]^{(i)}}$ of skew-product type such that
\[(Y_e)_{|e|\leq k} \stackrel{\rm{law}}{=} \hat{f}((U_e)_{|e|\leq k}),\]
where $(U_e)_{|e| \leq k}$ is an i.i.d. $\sim\rm{U}[0,1)$ array.

Now consider the coupling $(U_e,X_e)_{|e|\leq k}$ whose law is the relatively independent product over the condition $(Y_e)_{|e|\leq k} = \hat{f}((U_e)_{|e|\leq k})$:
\begin{multline}\label{eq:rel-prod}
\sfP((U_e)_e \in \d \bs{u},\,(X_e)_e \in \d \bs{a})\\
 = \sfP((U_e)_e \in \d \bs{u})\cdot \sfP((X_e)_e \in \d \bs{a}\,|\,(Y_e)_e = \hat{f}(\bs{u})).
\end{multline}

This is exchangeable and has the three desired properties.  The exchangeability follows because both factors on the right-hand side of~(\ref{eq:rel-prod}) are invariant under the action of $S_\bbN$ on the indexing set $e$, by the exchangeability of $(U_e)_e$ and $(Y_e,X_e)_e$. The first two of the properties listed are obvious, and the third follows from Proposition~\ref{prop:replicas} because the above relative product formula gives
\[\sfP\big((X_e)_e \in \d\bs{a}\,\big|\,(U_e)_e = \bs{u}\big) = \sfP\big((X_e)_e \in \d\bs{a}\,\big|\,(Y_e)_e = \hat{f}(\bs{u})\big),\]
and we conditioned on the equality $(Y_e)_e = \hat{f}((U_e)_e)$.
\end{proof}

\begin{proof}[Proof of Theorem A]\quad Let the process $(U_e,X_e)_{|e|\leq k}$ be as in the preceding corollary.  Applying Theorem~\ref{thm:Kal1} to this whole process gives functions $\hat{g}:[0,1)^{\P[k]}\to [0,1)^{\P[k]}$ and $\hat{h}:[0,1)^{\P[k]}\to \prod_{i\leq k}A_i^{[k]^{(i)}}$ of skew-product type such that
\begin{eqnarray}\label{eq:f-and-g}
((U_e)_{|e|\leq k},(X_e)_{|e|\leq k}) \stackrel{\rm{law}}{=} \big(\hat{g}((U'_e)_{|e|\leq k}),\hat{h}((U'_e)_{|e|\leq k})\big),
\end{eqnarray}
where again $(U'_e)_{|e|\leq k}$ are i.i.d. $\sim \rm{U}[0,1)$.

For the first coordinates, this reads
\[(U_e)_{|e|\leq k} \stackrel{\rm{law}}{=} \hat{g}((U'_e)_{|e|\leq k}).\]
Since both input and output are i.i.d. $\rm{U}[0,1)$ arrays, we may apply the Equivalence Theorem~\ref{thm:Kal2} to this equality of laws: it gives functions $G,G':[0,1)^{\P[k]}\to [0,1)^{\P[k]}$ of skew-product type whose skew-product tuples are Lebesgue-measure-preserving and which make the following diagram commute:
\begin{center}
$\phantom{i}$\xymatrix{
 [0,1)^{\P[k]}\ar_G[dd]\ar^{G'}[dr]\\
& [0,1)^{\P[k]}\ar^{\hat{g}}[dl]\\
[0,1)^{\P[k]}.
}
\end{center}
(Note that this seems almost unnecessary, since $\hat{g}$ already sends Lebesgue measure on $[0,1)^{\P[k]}$ to itself.  However, we will need the structure of Lebsgue-measure-preserving skew-product tuples, which need not follow in case $\hat{g}$ is not injective: see the remarks immediately preceding Theorem~\ref{thm:Kal2}.)

Now applying Corollary~\ref{cor:rep-skew-prod} to $G$, one obtains a function $H:([0,1)\times [0,1))^{\P[k]}\to [0,1)^{\P[k]}$ of skew-product type, given by a Lebesgue-measure-preserving skew-product tuple, and such that the above commutative diagram can be enlarged to
\begin{center}
$\phantom{i}$\xymatrix{
([0,1)\times [0,1))^{\P[k]}\ar^-H[rr]\ar_\Pi[ddrr] && [0,1)^{\P[k]}\ar_G[dd]\ar^{G'}[dr]\\
&&& [0,1)^{\P[k]}\ar^{\hat{g}}[dl]\\
&&[0,1)^{\P[k]}.
}
\end{center}

Let $(V_e)_{|e|\leq k}$ be another collection of i.i.d. $\rm{U}[0,1)$-r.v.s independent from $(U_e)_{|e|\leq k}$, and let $\hat{f} := \hat{h}\circ G' \circ H$.  Then the above diagram implies that:
\begin{itemize}
\item on the one hand,
\[(U_e)_{|e|\leq k} = \hat{g}(G'(H((U_e,V_e)_{|e|\leq k}))) \quad \hbox{a.s.},\]
\item and on the other, 
$G'(H((U_e,V_e)_{|e|\leq k}))$ is an i.i.d. array of $\rm{U}[0,1)$-r.v.s, and so
\begin{multline*}
(\hat{g}((U'_e)_{|e|\leq k}),\hat{h}((U'_e)_{|e|\leq k}))\\ \stackrel{\rm{law}}{=} 
\big(\hat{g}(G'(H((U_e,V_e)_{|e|\leq k}))),\hat{h}(G'(H((U_e,V_e)_{|e|\leq k})))\big).
\end{multline*}
\end{itemize}

Combining~(\ref{eq:f-and-g}) with these two facts now gives
\begin{eqnarray*}
(U_e,X_e)_{|e|\leq k} &\stackrel{\rm{law}}{=}&
\big(\hat{g}(G'(H((U_e,V_e)_{|e|\leq k}))),\hat{h}(G'(H((U_e,V_e)_{|e|\leq k})))\big)_{|e|\leq k}\\
&\stackrel{\rm{law}}{=}& \big((U_e)_{|e|\leq k},\hat{f}((U_e,V_e)_{|e|\leq k})\big),
\end{eqnarray*}
and conditioning both sides of this on $(U_e)_{|e|\leq k}$ gives
\[\sfP((X_e)_{|e|\leq k}\in \cdot\ |\,(U_e)_{|e|\leq k}) = \sfP(\hat{f}((U_e,V_e)_{|e|\leq k})\in \cdot\ |\,(U_e)_{|e|\leq k}),\]
as required.
\end{proof}

\section{Relation to Dovbysh-Sudakov Theorem}

\begin{proof}[Proof of Dovbysh-Sudakov Theorem]
The trick to this is the standard one-to-one correspondence
\[\big\{\hbox{PSD $(\bbN\times \bbN)$-matrices}\big\} \leftrightarrow \big\{\hbox{Gaussian measures on $\bbR^\bbN$}\big\}\]
in which a Gaussian measure is identified with its variance-covariance matrix. (This is elementary for finite PSD matrices, and then the infinite case follows by the Daniell-Kolmogorov Theorem: see~\cite[Theorem 6.14]{Kal02}.)  Because Gaussian measures are uniquely determined by their variance-covariance matrices, this correspondence intertwines the two permutations actions of $\bbN$, so from $(R_{ij})_{i,j}$ we may construct an ERM $\bs{\mu}$ on $\bbR^\bbN$ which is almost surely Gaussian, and such that
\[R_{ij} = \int_{\bbR^\bbN} x_i x_j\,\bs{\mu}(\d(x_n)_{n\in\bbN}) \quad \hbox{a.s.}.\]
Now Theorem B gives a representation
\[\bs{\mu} \stackrel{\rm{law}}{=} \int_0^1 \bigotimes_i\bs{\l}_i(t,\cdot)\,\d t\]
with $(\bs{\l}_i)_i$ drawn from some exchangeable sequence taking values in $B([0,1),\Pr\,\bbR)$.  Substituting this above gives
\[R_{ii} \stackrel{\rm{law}}{=} \int_0^1 \int_\bbR x^2\,\bs{\l}_i(t,\d x)\,\d t\]
and
\[R_{ij} \stackrel{\rm{law}}{=} \int_0^1 \big(\int_\bbR x\,\bs{\l}_i(t,\d x)\big)\big(\int_\bbR x\,\bs{\l}_j(t,\d x)\big)\,\d t \quad \hbox{when}\ i \neq j.\]
Letting
\[\frH = L^2([0,1),\d t),\]
\[\xi_i(t) = \int_\bbR x\,\bs{\l}_i(t,\d x)\]
and
\[a_i = \int_0^1 \Big(\int_\bbR x^2\,\bs{\l}_i(t,\d x) - \Big(\int_\bbR x\,\bs{\l}_i(t,\d x)\Big)^2\Big)\,\d t,\]
this is the desired representation.

(Note that $\xi_i$ must be in $\frH$ a.s. because
\[\int_0^1 \xi_i(t)^2\,\d t = \int_0^1\Big(\int_\bbR x\,\bs{\l}_i(t,\d x)\Big)^2\,\d t \leq \int_0^1\int_\bbR x^2\,\bs{\l}_i(t,\d x)\,\d t \stackrel{\rm{law}}{=} R_{ii},\]
which is finite a.s.)
\end{proof}

\section{Limiting behaviour of the Viana-Bray model}

Our second, and much more tentative, application for ERMs is to the study of the Viana-Bray (`VB') model~\cite{ViaBra85}.  This is the basic `dilute' mean-field spin glass model.  On the configuration space $\{-1,1\}^N$, it is given by the random Hamiltonian
\begin{eqnarray}\label{eq:VBHam}
H_N(\s) = \sum_{k=1}^M J_k\s_{i_k}\s_{j_k},
\end{eqnarray}
where:
\begin{itemize}
\item $M$ is a Poisson r.v. with mean $\a N$ (the thermodynamic limit is be taken with $\a$ fixed);
\item $i_1$, $j_1$, $i_2$, $j_2$, \ldots are indices from $[N]$ chosen uniformly and independently at random;
\item and $J_1$, $J_2$, \ldots are i.i.d. symmetric $\bbR$-valued r.v.s with some given distribution, often taken to be uniform $\pm 1$.
\end{itemize}
(There are many essentially equivalent variants of this model, but this popular version will do here.) From a quenched choice (that is, a fixed sample) of the random function $H_N$, the objects of interest are the resulting Gibbs measure
\[\g_{\b,N}\{\s\} = \frac{1}{Z_N(\b)}\exp (- \b H_N(\s)),\]
the partition function
\[Z_N(\b) = \sum_\s\exp(-\b H_N(\s))\]
and the expected specific free energy
\begin{eqnarray}\label{eq:expspec}
F_N(\b) = \frac{1}{N}\sfE\log Z_N(\b),
\end{eqnarray}
where the expectation is over the random function $H_N$.  We will sometimes drop the subscript `$\b$' or `$N$' in the sequel.

This is a relative of the older Sherrington-Kirkpatrick (`SK') model~\cite{MezParVir--book}, in which all pairs of spins $ij$ interact according to independent random coefficients $g_{ij}\sim \rm{N}(0,1/N)$.  The rigorous study of the SK model has become quite advanced in recent years; we will not credit all of the important contributions, but refer the reader to the books~\cite{Tal--SGbook,Pan--book} and the many references given there.  By contrast, most properties of the VB model remain conjectural.

A key tool in the study of the SK model is the use of random measures on Hilbert space as a kind of `limit object' for the random Gibbs measures $\g_{\b,N}$ as $N\to\infty$.  Viewing $\frac{1}{\sqrt{N}}\{-1,1\}^N$ as a subset of $\ell_2^N$, $\g_{\b,N}$ is itself a random Hilbert space measure, and the appropriate notion of convergence is convergence in distribution of the Gram-de Finetti matrices obtained by sampling.   Having taken a limit in this sense, a limit object in the form of a random measure on Hilbert space is provided by the Dovbysh-Sudakov Theorem.  This use of exchangeability and limit objects originates in works of Arguin~\cite{Arg08} and Arguin and Aizenman~\cite{ArgAiz09}, with a precedent in the study of classical mean-field models in the work~\cite{FanSpoVer80} of Fannes, Spohn and Verbeure.  It is explained in more detail in~\cite{Pan--book}.

The key point for this use of random Hilbert-space measures is that the main properties of the SK model, such as the free energy, really depend only on the covariances among the random variables $H(\s)$, and hence on this Hilbert space structure.  This is no longer true for the VB model, so a more refined tool is needed.  One possibility has been explored in~\cite{Pan--dil}, and before that physicists and mathematicians had already worked with the related notion of `multi-overlap structures' (see, e.g.,~\cite{DeS04,DeSFra09}, and also~\cite{PanTal04}, although the latter does not use that terminology).

Here we will simply propose exchangeable random measures as a fairly intuitive equivalent formalism, and compare it with two predecessors from the literature: the weighting schemes used by Panchenko and Talagrand in~\cite{PanTal04}, and Panchenko's use of directing functions in~\cite{Pan--dil}.  After introducing our notion of `limit object', we will give a fairly brisk summary of the translations between these formalisms; the calculations are all routine.  We will restrict attention to the Viana-Bray model as above for simplicity, but the discussion could easily be extended to a more general class of dilute models, as in~\cite{PanTal04,Pan--dil}.

\subsection{Basic idea}

If $\g_{\b,N}$ is as above, then it defines an ERM $\bs{\mu}$ by sampling: first quench the random measure $\g_{\b,N}$; then select samples (called `replicas') $\s^1$, $\s^2$, \ldots $\in \{-1,1\}^N$ i.i.d. $\sim \g_N$; and finally use these to define $\bs{\mu}$ as a mixture of delta masses:
\begin{eqnarray}\label{eq:emp}
\bs{\mu} = \frac{1}{N}\sum_{n=1}^N\delta_{(\s^1_n,\s^2_n,\ldots)}.
\end{eqnarray}
Identifying $\pm 1$ with the extreme points of $\Pr\{-1,1\}$, this is clearly a mixture of ERPMs of the kind considered previously.  Let $\Samp(\g_{\b,N})$ be the law of $\bs{\mu}$.

It now makes sense to say that $\g_{\b,N}$ \textbf{sampling converges} to some random probability measure $\bs{\g}_\b$ on $B([0,1),\Pr\{-1,1\})$ if $\Samp(\g_{\b,N})$ converges to $\Samp(\bs{\g}_\b)$ for the vague topology on $\Pr(\Pr\{-1,1\}^\bbN)$.  This last space is compact, and the laws of exchangeable random measures clearly comprise a further subspace which is closed for the vague topology (since invariance under any given continuous transformation of $\{-1,1\}^\bbN$ is a closed property).  Therefore one can always at least take subsequential limits of $(\Samp(\g_{\b,N}))_N$, and now Theorem B promises the existence of some $\bs{\g}_\b$ that represents the limiting ERM (although it is unique only up to equivalence).

This idea generalizes the more classical use of Gram-de Finetti matrices and their limits recalled above.  Starting from the SK model, the associated Gram-de Finetti matrix is obtained by sampling and then quenching the Gibbs measure, and then sampling from that Gibbs measure a sequence of states in $\{-1,1\}^N$ and computing their inner products as elements of $\ell_2^N$, normalized by $N$ (that is, their `overlaps').  Comparing with the random measure $\bs{\mu}$ in~(\ref{eq:emp}), this Gram-de Finetti matrix may be recovered as simply the (random) matrix of covariances of the different coordinates in $\{-1,1\}^\bbN$ under this (random) measure.

\subsection{Comparison with weighting schemes and directing functions}

In~\cite{PanTal04} the authors do not introduce a notion of limits as such for the random measures $\g_{\b,N}$, but they do formulate their most general results (Section 3 of that paper) in terms of some data that they call a `weighting scheme'.  This consists of:
\begin{itemize}
\item a sequence of $\bbR$-valued r.v.s $(X_k)_k$, and a family $((X^{i,j}_k)_k)_{i,j}$ of i.i.d. copies of this sequence indexed by $(i,j) \in \bbN^2$;
\item and, independently of these, a $[0,1]$-valued random sequence of weights $(v_k)_k$ such that $\sum_k v_k = 1$.
\end{itemize}

These data appear in an upper-bound formula for the free energy which will be recalled below.  They can be encapsulated in a certain directing random measure $\bs{\g}$ on $B([0,1),\Pr\{-1,1\})$ as follows.  First, identifying elements of $\Pr\{-1,1\}$ with their expectations gives
\[B([0,1),\Pr\{-1,1\}) = B([0,1),[-1,1]).\]
Now let $\Phi(x):= \rm{e}^x/(\rm{e}^x + \rm{e}^{-x})$. Applying Lemma~\ref{lem:NOL}, we may find a sequence $(f_k)_k$ in $B([0,1),[-1,1])$ such that
\begin{eqnarray}\label{eq:def-f}
(\Phi(X_k))_k \stackrel{\rm{law}}{=} (f_k(U))_k \quad \hbox{when}\ U\sim \rm{U}[0,1).
\end{eqnarray}
To finish, let $\bs{\g}$ be the atomic random measure
\begin{eqnarray}\label{eq:WS}
\bs{\g} = \sum_{k\geq 1}v_k\delta_{f_k},
\end{eqnarray}
so the randomness of $\bs{\g}$ is derived from the random choice of the weights $v_k$.

Clearly one could find many other ways to convert a weighting scheme into an ERM, but this translation is appropriate because it gives the correct correspondence between upper-bound formulae for the free energy to be recalled below.

On the other hand, in~\cite{Pan--dil} Panchenko does introduce a family of limit objects, closely related to our use of limiting ERMs.  Given the random Gibbs measure $\g_{\b,N}$ on $\{-1,1\}^N$, he draws independent replicas $\s^1$, $\s^2$, \ldots from it and then considers the joint distribution of the whole $(N\times \infty)$-indexed, $\{-1,1\}$-valued random array
\[(\s^\ell_n)_{1 \leq n \leq N,\,\ell \geq 1}.\]

Whereas we used these replicas to form an empirical measure which is an ERM, Panchenko chooses an arbitrary extension of this to a two-dimensional infinite random array.  Letting $N\to\infty$, if one considers a subsequence of the $\g_N$ for which these joint distributions converge, then in the limit one obtains a random $\{-1,1\}$-valued array which is row-column exchangeable.  Applying Theorem~\ref{thm:RCE}, this array has the same law as
\[\big(\s(U,U_n,V_\ell,W_{n\ell})\big)_{n,\ell\geq 1}\]
for some measurable function $\s:[0,1)^4\to \{-1,1\}$, where $U$, $U_n$ for $n\geq 1$, $V_\ell$ for $\ell \geq 1$ and $W_{n\ell}$ for $n,\ell\geq 1$ are i.i.d. $\sim\rm{U}[0,1)$.

Panchenko then uses $\s$ itself as his limit object for the sequence $(\g_{\b,N})_N$.  The equivalence between this formalism and the use of directing random measures on $B([0,1),\Pr\{-1,1\})$ is just the equivalence between our Theorem B and Theorem~\ref{thm:RCE} described in Subsection~\ref{subs:RCE} above.

\subsection{Formula for the limiting free energy}

A central result of~\cite{Pan--dil} is a formula for the asymptotic expected free energy of models such as~(\ref{eq:VBHam}) in terms of a functional of the directing functions introduced above: see~\cite[Theorem 2]{Pan--dil}.  For the VB model itself the result is as follows.

\begin{thm}[Free energy formula]\label{thm:var}
As $N\to\infty$, the expected specific free energy from~(\ref{eq:expspec}) satisfies
\[\lim_{N\to\infty}F_N = \inf_{\s}\P(\s),\]
where for $\s:[0,1)^4\to\{-1,1\}$ we have
\begin{multline*}
\P(\s) := \log 2 + \sfE^{(1)}\log\sfE^{(2)}\Big(\cosh\b\sum_{i=1}^{K_1}J_i \s(W,U,V_i,X_i)\Big)\\
- \sfE^{(1)}\log\sfE^{(2)}\Big(\exp\b\sum_{i=1}^{K_2}J_i \s(W,U,V_i,X_i)\s(W,U,V'_i,X'_i)\Big),
\end{multline*}
where:
\begin{itemize}
\item all the r.v.s $W$, $U$, $V_1$, $V_2$, \ldots, $V_1'$, $V'_2$, \ldots, $X_1$, $X_2$, \ldots, $X_1'$, $X'_2$, \ldots are i.i.d. $\sim\rm{U}[0,1)$,
\item $K_1$ is an independent Poisson r.v. of mean $2\a$,
\item $K_2$ is an independent Poisson r.v. of mean $\a$,
\item and the coefficients $J_i$ are chosen independently from the same distribution as before,
\end{itemize}
and where
\[\sfE^{(1)} = \hbox{expectation over}\ W,K_1,K_2,(V_i)_i,(V'_i)_i\ \hbox{and}\ (J_i)_i\]
and
\[\sfE^{(2)} = \hbox{expectation over}\ U,(X_i)_i\ \hbox{and}\ (X'_i)_i.\]
 \qed
\end{thm}

If $\bs{\g}$ is the random directing measure on $B([0,1),\Pr\{-1,1\})$ that corresponds to $\s$, then the above formula may easily be recast in terms of $\bs{\g}$: it is
\begin{multline*}
\log 2 + \sfE\log \int_B \sum_{\eps_1,\ldots,\eps_{K_1} = \pm 1}\prod_{i=1}^{K_1}f(V_i,\{\eps_i\})\Big(\cosh \b\sum_{i=1}^{K_1} J_i\eps_i\Big)\ \bs{\g}(\d f)\\
 - \sfE\log \int_B \sum_{\scriptsize{\begin{array}{c}\eps_1,\ldots,\eps_{K_2} = \pm 1\\ \eps'_1,\ldots,\eps'_{K_2} = \pm 1\end{array}}}\prod_{i=1}^{K_2}f(V_i,\{\eps_i\})f(V'_i,\{\eps'_i\})\Big(\exp \b\sum_{i=1}^{K_2} J_i\eps_i\eps'_i\Big)\ \bs{\g}(\d f),
\end{multline*}
where
\[B = B([0,1),\Pr\{-1,1\}),\]
and where $\sfE$ is now the expectation over all the random data $\g$, $K_1$, $K_2$, $(V_i)_i$, $(V'_i)_i$ and $(J_i)_i$.

Another elementary (but tedious) calculation shows that under the correspondence~(\ref{eq:WS}) this coincides with the upper-bound expression that appears in~\cite{PanTal04}: the right-hand side of inequality (3.3) in that paper.  It is for the sake of this calculation that one uses the function $\Phi$ to define $f_k$ in~(\ref{eq:def-f}).

\begin{rmk}
In~\cite{Pan--dil} Panchenko also shows that the quantity above is unchanged if one instead takes the infimum only over those directing functions $\s$ that satisfy an analog of the Aizenman-Contucci stability under cavity dynamics.  This modification could also easily be formulated in terms of random directing functions, but we omit it for the sake of brevity. \fin
\end{rmk}

\subsection{The analog of ultrametricity}

After the general formalism of Section 3 of~\cite{PanTal04}, Sections 4 and 5 of that paper propose a special class of weighting scheme objects that correspond to the physicists' notion of `replica-symmetry breaking', and conjecture that these give the correct expression for the limiting free energy.  Following the prescriptions of the preceding subsections, we can translate this conjecture into a proposal for a class of limiting random directing measures which adapt the classical Parisi ultrametricity ansatz~\cite{Pan--book} to the setting of dilute models.  As before, the necessary calculations are simple but tedious, so we omit the details.  Some discussion along these lines is given in~\cite{Pan--dil} for the SK model, rather than for dilute models.

The key objects seem to be the following.  Suppose that $T$ is a discrete rooted tree with all leaves at a fixed finite distance from the root.  (The discussion that follows can certainly be extended to more general trees, but we omit that here.)  Let $\ast$ be the root and $\partial T$ the set of leaves.  Also, let $\S$ be the Borel $\s$-algebra of $[0,1)$.  We formulate the following on $[0,1)$, but it clearly makes sense on any probability space.

\begin{dfn}
A \textbf{branching filtration on $([0,1),\S,\rm{Leb})$ indexed by $T$} is a family of $\s$-subalgebras $(\S_t)_{t\in T}$ such that
\begin{itemize}
\item $t \leq t'$ $\Longrightarrow$ $\S_t \subseteq \S_{t'}$;
\item for any $t_0,\ldots,t_m$, the $\s$-algebra $\S_{t_0}$ is conditionally independent from $\S_{t_1} \vee \cdots \vee \S_{t_m}$ over $\S_s$ where $s = (t_0\wedge t_1) \vee (t_0\wedge t_2) \vee \cdots \vee (t_0\wedge t_m)$, the closest vertex of $T$ to $t_0$ which is a common ancestor of $t_0$ and some other $t_i$.
\end{itemize}
By analogy with ordinary filtrations, the branching filtration is \textbf{complete} if every $\S_t$ is complete for Lebesgue measure.

Given a branching filtration $\bs{\S} = (\S_t)_{t \in T}$, a \textbf{branchingale adapted to $\bs{\S}$} is a family of integrable $\bbR$-valued functions $(f_t)_{t \in T}$ on $[0,1)$ such that
\begin{itemize}
\item $f_t$ is $\S_t$-measurable;
\item $t \leq t'$ $\Longrightarrow$ $f_t = \sfE(f_{t'}\,|\,\S_t)$.
\end{itemize}
Observe that in this case every root-leaf path $\ast v_1v_2\cdots v_d$ gives a martingale $(f_\ast,f_{v_1},\ldots,f_{v_r})$ adapted to the filtration $(\S_\ast,\S_{v_1},\ldots,\S_{v_d})$; we call the branchingale \textbf{homogeneous} if every root-leaf path gives a martingale with the same distribution.
\end{dfn}

Sometimes we refer to the whole collection $(f_t,\S_t)_{t\in T}$ as a branchingale.

\begin{rmk}
Of course, stochastic processes indexed by trees have been studied before, but I have not been able to find a reference for precisely this notion.  Much of the literature concerns tree-indexed Markov processes, as in~\cite{BenPer94}, but I do not see why the r.v.s $f_v$ that we will use should have the Markov property. \fin
\end{rmk}

\begin{dfn}
A subset $Y\subseteq B([0,1),[-1,1])$ is \textbf{hierarchically distributed} if it equals $\{f_v:\ v \in \partial T\}$ for some \emph{homogeneous} branchingale $(f_t,\S_t)_{t \in T}$.  The minimal depth of $T$ in such a representation is the \textbf{depth} of the set $Y$.
\end{dfn}

Now a simple calculation shows that under the correspondence~(\ref{eq:WS}), the special weighting schemes used to formulate the $r$-step replica-symmetry breaking bound in~\cite[Section 5]{PanTal04} correspond to random measures $\bs{\g}$ which are a.s. supported on hierarchically distributed sets of depth $r$, and with the weights given by a Derrida-Ruelle probability cascade that follows the indexing tree.

To be more specific, in their work they consider r.v.s $X_t$ indexed by the leaves $t$ of a tree $T$ of depth $r$ and infinite branching, and specify their joint distribution by constructing a larger family of random variables
\[(\eta^{(0)},\eta^{(1)}_{t_1},\eta^{(2)}_{t_1t_2},\ldots,\eta^{(r-1)}_{t_1t_2\cdots t_{r-1}},\eta^{(r)}_{t_1t_2\cdots t_r})\]
indexed by all downwards paths from the root in $T$, where:
\begin{itemize}
 \item $\eta^{(r)}_{t_1\cdots t_r} = X_{t_r}$ for each leaf $t_r \in \partial T$,
\item for a shorter path $t_1t_2\cdots t_s$, $0 \leq s \leq r-1$, the r.v. $\eta^{(s)}_{t_1t_2\cdots t_s}$ takes values in the space
\[\underbrace{\Pr(\Pr(\cdots\Pr(}_{r-s}\bbR))),\]
\item and for each $t_1t_2\cdots t_s$ with $s \leq r-1$, the r.v.s $\eta^{(s+1)}_{t_1t_2\cdots t_st}$ indexed by all the children $t$ of $t_s$ are chosen independently from $\eta^{(s)}_{t_1t_2\cdots t_s}$, and similarly the random variables at all further children along distinct ancestral lines are conditionally independent.
\end{itemize}
Such a structure arises from a homogeneous branchingale $(f_t,\S_t)_{t \in T}$ for which $0 < f_t < 1$ a.s. as follows. Let $\eta^{(r-1)}_{t_1\cdots t_{r-1}}$ be the conditional distribution of $\Phi^{-1}\circ f_{t_r}$ on $\S_{t_{r-1}}$ for any child $t_r$ of $t_{r-1}$, where $\Phi(x) = \rm{e}^x/(\rm{e}^x + \rm{e}^{-x})$ as before, and the condition $0 < f_{t_r} < 1$ ensures that this composition is defined a.s..  Now let $\eta^{(r-2)}_{t_1\cdots t_{r-2}}$ be the conditional distribution of $\eta^{(r-1)}_{t_1\cdots t_{r-1}}$ on $\S_{t_{r-2}}$, and so on.  These are then related to the functions $f_t$ themselves in that $f_{t_s}$ is obtained from $\eta^{(s)}_{t_1\cdots t_s}$ by applying $\Phi$ and then taking barycentres $r-s$ times.  If one starts instead from the r.v.s $\eta^{(s)}_{t_1\cdots t_s}$ as above, another simple (but lengthy) iterated appeal to Lemma~\ref{lem:NOL} produces a homogeneous branchingale that gives rise to it.

Thus, the natural analog of the Parisi ultrametricity ansatz for the Viana-Bray model seems to be that in the infimum of Theorem~\ref{thm:var}, if one formulates the right-hand side in terms of directing random measures, it is enough to consider directing random measures that are a.s. supported on hierarchically distributed subsets of $B([0,1),\Pr\{-1,1\})$.

\bibliographystyle{abbrv}
\bibliography{bibfile}

\end{document}